\documentclass{article}
\usepackage[all,ps,arc]{xy}
\usepackage{amsmath,amssymb,latexsym}
\usepackage{amsfonts}
\usepackage{amsthm}
\usepackage{graphicx}

\newtheorem{Theorem}{Theorem}[section]
\newtheorem{Lemma}[Theorem]{Lemma}
\newtheorem{Proposition}[Theorem]{Proposition}
\newtheorem{Definition}[Theorem]{Definition}
\newtheorem{Corollary}[Theorem]{Corollary}
\theoremstyle{definition}\newtheorem{Remark}[Theorem]{Remark}
\newtheorem{Example}[Theorem]{Example}

\DeclareMathOperator{\Ker}{Ker}
\DeclareMathOperator{\Coker}{Coker}
\DeclareMathOperator{\coker}{coker}
\DeclareMathOperator{\Act}{Act}
\DeclareMathOperator{\USGA}{USGA}

\newcommand{\C}{\ensuremath{\mathcal{C}}}

\newcommand{\LACC}{{\rm (LACC)}}
\newcommand{\CS}{{\rm (CS)}}
\newcommand{\SH}{{\rm (SH)}}

\newcommand{\Gp}{\ensuremath{\mathsf{Gp}}}

\newcommand{\Pt}{\ensuremath{\mathsf{Pt}}}
\newcommand{\RG}{\ensuremath{\mathsf{RG}}}
\newcommand{\PXMod}{\ensuremath{\mathsf{PXMod}}}
\newcommand{\PX}{\ensuremath{{_\mathsf{PX}}}}
\newcommand{\XMod}{\ensuremath{\mathsf{XMod}}}

\newcommand{\plus}{\ensuremath{+_\PX}}
\newcommand{\farf}{\mathrel{\text{\scalebox{.85}{$\Join$}}}}
\newcommand{\dense}{\ensuremath{\mathsf{PXMod}_B(\C)}}
\newcommand{\boxslash}{\mathrel{\vcenter{\hbox{\hskip.1ex$\Box$}\vskip-3.2ex\hbox{\rotatebox{-8}{$\smallsetminus$}}}}}
\newcommand{\blslash}{\reflectbox{$\boxslash$}}
\newcommand{\brslash}{$\boxslash$}
\newcommand{\btimes}{\scalebox{.85}{$\boxtimes$}}
\newcommand{\twcl}{\text{$\,\raisebox{-.2ex}{\rotatebox{45}{$\scriptstyle\boxminus$}}\,$}}
\newcommand{\twcr}{\text{$\,\raisebox{-.2ex}{\reflectbox{\rotatebox{45}{$\scriptstyle\boxminus$}}}\,$}}
\newcommand{\twc}{\text{$\,\raisebox{-.2ex}{\rotatebox{45}{$\scriptstyle\boxplus$}}\,$}}
\newcommand{\twcpxl}{\text{$\,\raisebox{-.05ex}{\blslash}\,$}}
\newcommand{\twcpxr}{\text{$\,\raisebox{-.05ex}{\brslash}\,$}}
\newcommand{\twcpx}{\text{$\;\raisebox{.05ex}\btimes\;$}}

\newdir{|>}{{}*!/8pt/@{|}*!/3.5pt/:(1,-.2)@^{>}*!/3.5pt/:(1,+.2)@_{>}}
\newdir{ >}{{}*!/-10pt/@{>}}
\newdir{ |>}{!/10pt/{}*!/4.5pt/@{|}*:(1,-.2)@^{>}*:(1,+.2)@_{>}}

\begin{document}

\newenvironment{changemargin}[2]{\begin{list}{}{
\setlength{\topsep}{0pt}
\setlength{\leftmargin}{0pt}
\setlength{\rightmargin}{0pt}
\setlength{\listparindent}{\parindent}
\setlength{\itemindent}{\parindent}
\setlength{\parsep}{0pt plus 1pt}
\addtolength{\leftmargin}{#1}\addtolength{\rightmargin}{#2}
}\item}{\end{list}}

\title{Peiffer product and Peiffer commutator for internal pre-crossed modules}
\author{Alan S. Cigoli, Sandra Mantovani and Giuseppe Metere}

\maketitle

\begin{abstract}
In this work we introduce the notions of Peiffer product and Peiffer commutator of internal pre-crossed modules over a fixed object $B$, extending the corresponding classical notions to any semi-abelian category $\C$. We prove that, under mild additional assumptions on $\C$, crossed modules are characterized as those pre-crossed modules $X$ whose Peiffer commutator $\langle X,X \rangle$ is trivial. Furthermore we provide suitable conditions on $\C$ (fulfilled by a large class of algebraic varietes, including among others groups, associative algebras, Lie and Leibniz algebras) under which the Peiffer product realizes the coproduct in the category of crossed modules over $B$.
\end{abstract}


\section{Introduction}

Let $B$ be a group. A pre-crossed module $(X,\delta)$ over $B$ is a group homomorphism $\delta\colon X \to B$ together with a (left) action of $B$ on $X$ such that, for $b$ in $B$ and $x$ in $X$:
$$ \delta(^b x)=b\delta(x)b^{-1} \,. $$
A pre-crossed module is called crossed module if, in addition, the following \emph{Peiffer identity} holds for any $x$, $x'$ in $X$:
$$ ^{\delta x}x' = xx'x^{-1} \,. $$
Crossed modules have been introduced by Whitehead in \cite{Whi46} while working on homotopy groups (see also \cite{Whi41} and \cite{Peiffer}). Thereafter, they have proved to be a useful technical tool in many areas of mathematics, including homotopy theory and homological algebra. An extensive recollection of results concerning crossed modules can be found in the on-going project \cite{Menage}, while a detailed account on their applications in algebraic topology is in \cite{NAT}.

Pre-crossed $B$-modules organize in a category $\PXMod_B(\Gp)$ where crossed modules form a full reflective subcategory
$$ \xymatrix{\PXMod_B(\Gp) \ar@<1ex>[r] \ar@{}[r]|-\bot & \XMod_B(\Gp) \ar@<1ex>[l]} $$
The reflection of a pre-crossed module $X$ is obtained by imposing the Peiffer identities. The kernel of this quotient is known as the \emph{Peiffer subgroup} of $X$ and denoted by $\langle X,X \rangle$.

In order to describe algebraic invariants of relevant topological constructions, it is natural to study the corresponding notions in the category of crossed modules, as Brown does for coproducts in \cite{BrownCCM}. Namely, given two pre-crossed modules $X$ and $Y$ over the same group $B$, the free product $X \ast Y$ of the two groups naturally inherits a structure of pre-crossed $B$-module that makes it a coproduct of $X$ and $Y$ in the category $\PXMod_B(\Gp)$. Whenever $X$ and $Y$ are crossed modules, their coproduct in $\XMod_B(\Gp)$ is obtained as the reflection of $X \ast Y$. In the above cited paper \cite{BrownCCM}, Brown provides an alternative construction, by considering first the semidirect product $X \rtimes Y$ as a pre-crossed $B$-module and then reflecting it into $\XMod_B(\Gp)$. As observed by Gilbert and Higgins in \cite{GilHig}, this coproduct is actually given by what they call the \emph{Peiffer product} of $X$ and $Y$, denoted by $X \farf Y$, which they define more generally for two groups acting on each other in a compatible way. Explicitly, $X \farf Y$ is obtained from $X \ast Y$ by quotienting out the Peiffer words
$$ xyx^{-1}(^{\delta x}y)^{-1} \,, \qquad yxy^{-1}(^{\delta y}x)^{-1} \,, $$
with $x$ in $X$ and $y$ in $Y$.

More generally, in $\PXMod_B(\Gp)$, for $X$ and $Y$ pre-crossed submodules of a given pre-crossed module $A$, following the work \cite{CE} by Conduch\'e and Ellis, one can form the subgroup of $A$ generated by the Peiffer words defined as above. This is known as the \emph{Peiffer commutator} of $X$ and $Y$, and it is denoted by $\langle X,Y \rangle$. Of course, the Peiffer subgroup cited above is an instance of the latter. Let us observe that the Peiffer commutator generalizes the classical commutator of groups, which is recovered when $B$ is the trivial group. A calculus of Peiffer commutators and some applications to homology can be found in \cite{CE,BCPeiff}.

The problem of studying the coproduct construction is faced also for crossed modules of Lie algebras and associative algebras by means of  similar methods respectively in \cite{CL2000} and \cite{Shammu}. On the other hand, in any semi-abelian variety (including the cases above, together with many others, such as Leibniz, Jordan, Poisson algebras), by using a categorical approach, Everaert and Gran introduce in \cite{EG:XMod} a notion of Peiffer commutator $\langle X, A \rangle$ for the special case where $X$ is a normal pre-crossed submodule of $A$.

The aim of this work is to develop a structural approach to the study of Peiffer products and Peiffer commutators extending these notions to the case of internal pre-crossed and crossed modules in a semi-abelian category \C\ (see \cite{Janelidze}). As a result, we prove that, in any semi-abelian category satisfying the condition \SH\ (see e.g. \cite{MFVdL}), the reflection of pre-crossed modules onto crossed modules can be computed by means of the quotient over the Peiffer commutator.

Our approach reveals that the generalization of the Peiffer commutator and Peiffer product to the semi-abelian context is far from being a mere translation of the corresponding notions for groups, since in general the necessary constructions have to be performed in the category \dense\ rather than in \C.

This observation naturally leads to the task of detecting assumptions under which these new notions can be described directly  in the base category. The property of \C\ being \emph{algebraically coherent} \cite{CGrayVdL2} turns out to be crucial in this sense. A further assumption is taken in Section \ref{sec:coproduct} in order to prove that the Peiffer product yields the coproduct in the category of internal crossed $B$-modules. All of these results apply, for example, to \emph{categories of interest} in the sense of Orzech \cite{Orzech}, such as Lie algebras over a field $k$, rings and associative algebras in general, Poisson algebras, Leibniz algebras among others.

The paper is organized as follows. In Section \ref{sec:PXMod} we describe the category of internal pre-crossed modules over a fixed object $B$ and its equivalence with internal reflexive graphs over $B$. Moreover, we provide the description of some relevant limits and colimits in \dense\ in terms of universal constructions in the base category. Peiffer product and Peiffer commutator of internal pre-crossed $B$-modules are introduced in Section \ref{sec:Peiffer}. In Section \ref{sec:alg.coh} we reformulate, in the algebraically coherent context, the definitions given in Section \ref{sec:Peiffer} and we provide examples of explicit calculation of the Peiffer commutator. The last section is devoted to the proof of the fact that, under suitable assumptions, the Peiffer product is the coproduct in $\XMod_B(\C)$.


\section{The category \dense} \label{sec:PXMod}

In this section we recall the definition of internal pre-crossed module in a semi-abelian category \C. For the related notions and terminology, we refer to \cite{MaMe10-1}.

An internal pre-crossed module in \C\ is an arrow $\delta \colon X \to B$, together with an internal action  $\xi$ of $B$ on $X$ (see \cite{BJK}), making the following diagram commute:
$$ \xymatrix{
    B \flat X \ar[r]^{1 \flat \delta} \ar[d]_{\xi} & B \flat B \ar[d]^{\chi} \\
    X \ar[r]_{\delta} & B
} $$
where $\chi$ is the conjugation action of $B$ on itself.

The purpose of this section is to study some basic constructions in the category \dense\ of pre-crossed modules in \C\ with codomain a fixed object $B$. We widely use the equivalence between internal pre-crossed modules and internal reflexive graphs in \C\ (see \cite{Janelidze}), which associates with every reflexive graph:
$$ \xymatrix{
    X_1 \ar@<.8ex>[r]^{d} \ar@<-.8ex>[r]_{c} & B \ar[l]|{e}
} $$
its normalization $\delta \colon X \stackrel{c \cdot \ker (d)}{\longrightarrow} B$, equipped with the conjugation action of $B$ on $X$ computed in $X_1$. Conversely, every pre-crossed module $(\delta \colon X \to B, \xi)$ yields a reflexive graph:
$$ \xymatrix{
    X \rtimes B \ar@<.8ex>[r]^-{p} \ar@<-.8ex>[r]_-{[\delta,1\rangle} & B \ar[l]|-{i}
} $$
where $p$ and $i$ are the canonical projection and inclusion respectively and $[\delta,1\rangle$ is the unique arrow making the following diagram commute:
$$ \xymatrix{
    X \ar[r]^-{j} \ar[dr]_{\delta} & X \rtimes B \ar@{-->}[d]^(.4){[\delta,1\rangle} & B \ar[l]_-{i} \ar[dl]^{1} \\
    & B
} $$
Thanks to the equivalence mentioned above, every categorical construction in the sequentiable (in the sense of Bourn \cite{Bourn2001}, see \cite{EG:XMod}) category $\RG_B(\C)$ of internal reflexive graphs in \C\ with a given object of objects $B$ has a counterpart in \dense. We are going to describe how the main constructions we need for the present paper (factorizations, some limits and colimits) can be performed directly in the category \dense\ by means of constructions in the base category \C.

\subsection{The (regular epi, mono) factorization}

Let $f_1$ be an arrow in $\RG_B(\C)$. Its (regular epi, mono) factorization can be obtained by means of the (regular epi, mono) factorization $(q_1,m_1)$ of $f_1$ as an arrow in \C:
$$ \xymatrix@=7ex{
    X_1 \ar@<1ex>[d]^{d} \ar@<-1ex>[d]_{c} \ar@{->>}[r]^{q_1}
        & \bullet \ar@<1.5ex>[d]^{d'm_1} \ar@<-1.5ex>[d]_{c'm_1} \ar@{ >->}[r]^{m_1}
        & Y_1 \ar@<1ex>[d]^{d'} \ar@<-1ex>[d]_{c'} \\
    B \ar[u]|{e} \ar[r]_{1} & B \ar[u]|{q_1e} \ar[r]_{1} & B \ar[u]|{e'}
} $$
Indeed, $m_1$ is obviously a monomorphism in $\RG_B(\C)$ and $q_1$ is the coequalizer in $\RG_B(\C)$ of the following pair of morphisms:
$$ \xymatrix@=7ex{
    R \ar@<.5ex>[r]^{r_1} \ar@<-.5ex>[r]_{r_2} \ar@<2ex>[d]^{dr_1} \ar@<-2ex>[d]_{cr_1}
        & X_1 \ar@<-1ex>[d]_{c} \ar@<1ex>[d]^{d} \\
    B \ar[u]|{\langle e,e \rangle} \ar[r]_{1} & B \ar[u]|{e}
} $$
where $(R,r_1,r_2)$ is the kernel pair of $q_1$.

By equivalence, we have a (regular epi, mono) factorization in \dense\ by taking the restrictions $(q,m)$ to the kernels. Moreover, $q$ is a regular epimorphism in \C\ being a pullback of $q_1$, while $m$ is obviously a monomorphism. Thus the following result holds:

\begin{Proposition}
Every morphism in \dense\ has a (regular epi, mono) factorization which can be obtained by means of the (regular epi, mono) factorization of the corresponding arrow in \C:
$$ \xymatrix@=7ex{
    X \ar@{ ->>}[r]^-{q} \ar[dr]_{\delta_X} & \bullet \ar[d]^(.4){\delta_Y m} \ar@{ >->}[r]^{m}
        & Y \ar[dl]^{\delta_Y} \\
    & B
} $$
\end{Proposition}

We state here the following lemma, which will be used later on.

\begin{Lemma} \label{lem:3outof2}
Let $X$, $Y$ and $Z$ be pre-crossed $B$-modules in \C, $f\colon X \to Y$ and $g\colon Y \to Z$ arrows in \C\ such that the composite $g \cdot f$ is a pre-crossed module morphism. Then
\begin{enumerate}
    \item if $g$ is a monomorphism in \dense, $f$ is also a pre-crossed module morphism;
    \item if $f$ is a regular epimorphism in \dense, $g$ is also a pre-crossed module morphism.
\end{enumerate}
\end{Lemma}

\begin{proof}
It is easy to prove that in both cases $f$ and $g$ are morphisms in the slice category $\C\downarrow B$. It remains to prove that they are equivariant. Let us consider the following diagram, where the outer square commutes since the composite $g\cdot f$ is a pre-crossed module morphism:
$$ \xymatrix{
    B \flat X \ar[r]^{1\flat f} \ar[d]_{\xi_X} & B \flat Y \ar[r]^{1 \flat g} \ar[d]_{\xi_Y}
        & B \flat Z \ar[d]^{\xi_Z} \\
    X \ar[r]_f & Y \ar[r]_g & Z
} $$
When $g$ is a monomorphism of pre-crossed modules, the right hand square commutes and, by cancellation, the same holds for the one on the left. When $f$ is a regular epimorphism of pre-crossed modules, the left hand square commutes and, moreover, $1\flat f$ is also a regular epimorphism since these are preserved by the functor $B \flat -$ (see e.g.\ \cite{MaMe10-2}). Again, by cancellation, the right hand square commutes as desired.
\end{proof}

\subsection{Limits} \label{sec:limits}

It is well known that the category $\RG_B(\C)$ has pullbacks that are computed in the following way. Given a cospan in $\RG_B(\C)$:
$$ \xymatrix@=7ex{
    X_1 \ar@<1ex>[d]^{d} \ar@<-1ex>[d]_{c} \ar[r]^{f_1} & \bullet \ar@<1ex>[d] \ar@<-1ex>[d]
        & Y_1 \ar@<1ex>[d]^{d'} \ar@<-1ex>[d]_{c'} \ar[l]_{g_1} \\
    B \ar[u]|{e} \ar[r]_{1} & B \ar[u] \ar[r]_{1} & B \ar[u]|{e'}
} $$
it is easy to see that the pullback $X_1 \times_{(f_1,g_1)} Y_1$ is endowed with a reflexive graph structure over $B$. Since the kernel functor $\Ker_B \colon \Pt_B(\C) \to \C$ preserves limits, the restriction of this pullback to the kernels of the domain projections is again a pullback, yielding the following result.

\begin{Proposition}
Every cospan in \dense\ has a pullback which can be obtained by means of the pullback of the corresponding arrows in \C:
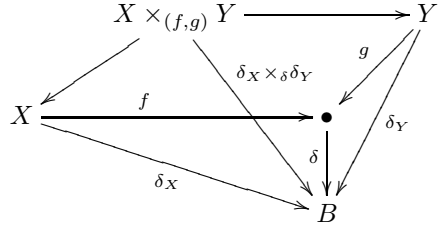
$$ \xymatrix{
    & X \times_{(f,g)} Y \ar[dl] \ar[rr] \ar[ddr]^(.35){\delta_X \times_\delta \delta_Y}
        & & Y \ar[ddl]^{\delta_Y} \ar[dl]_{g} \\
    X \ar[drr]_{\delta_X} \ar[rr]^(.4)f & & \bullet \ar[d]_(.4)\delta \\
    & & B
} $$
\end{Proposition}

Applying the last proposition to the particular case of a pullback along the initial map, we can describe explicitly how a kernel is computed in \dense.

\begin{Proposition}
A kernel diagram in \dense\ is a diagram of the following type:
$$ 
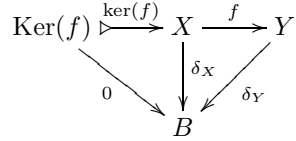
\xymatrix{
    \Ker(f) \ar@{ |>->}[r]^-{\ker(f)} \ar[dr]_0 & X \ar[d]^(.4){\delta_X} \ar[r]^{f} & Y \ar[dl]^{\delta_Y} \\
    & B
} $$
where the action of $B$ on $\Ker(f)$ is the restriction of the one on $X$.
\end{Proposition}

The following result essentially depends on the characterization of kernels in $\RG_B(\C)$, which is a special case of Proposition 6.2.1 in \cite{Borceux-Bourn}:

\begin{Proposition} \label{prop:normal}
The kernels in the category \dense\ are precisely the arrows of the form:
$$ 
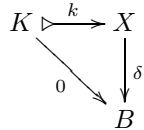
\xymatrix{
    K \ar@{ |>->}[r]^k \ar[dr]_0 & X \ar[d]^{\delta} \\
    & B
} $$
where $k$ is a kernel in \C\ and the action of $B$ on $X$ passes to the quotient $X/K$.
\end{Proposition}

In a semi-abelian category \C\ which is strongly protomodular the above condition on the $B$-action comes for free, so it suffices to ask that $k$ is a kernel. In fact, in the semi-abelian context, the condition ``every action which restricts to a kernel passes to the quotient'' is equivalent to strong protomodularity (see \cite{MeQoa} for details).
\medskip

\subsection{Colimits} \label{sec:colimits}

While limit constructions are quite easy in \dense, colimits can be rather complicated to describe, unless we assume some additional conditions on the base category.

It is well known that a pushout in $\RG_B(\C)$ is constructed by means of levelwise pushouts in \C. In general, it may not be easy to compute pushouts directly in \dense. An exception is given by the cokernel of a kernel. Following Bourn \cite{Bourn2001}, in the sequentiable context, we call cokernel of $k\colon X \to Y$ the pushout of $k$ along the unique arrow $X\to 0$ (if it exists, as in the case of kernels). Thanks to Proposition \ref{prop:normal}, it is easy to prove the following result.

\begin{Proposition}
The cokernel of a kernel $k$ in \dense\ is a diagram of the following type (the action of $B$ on $\Coker(k)$ is induced by the one on $X$):
$$ \xymatrix@C=7ex{
    K \ar@{ |>->}[r]^-{k} \ar[dr]_0 & X \ar[d]^{\delta} \ar@{ -|>}[r]^-{\coker(k)} & \Coker(k) \ar[dl]^{\delta'} \\
    & B
} $$
\end{Proposition}

Another manageable colimit is the pushout of two regular epimorphisms.

\begin{Proposition}
The pushout of two regular epimorphisms in \dense\ can be obtained
by means of the pushout of the corresponding arrows in \C:
$$ \xymatrix{
    \bullet \ar@{->>}[dr]^f \ar@{->>}[rr]^g \ar@/_/[ddr]_{\delta} & & Y \ar[ddl]^(.3){\delta_Y} \ar@{->>}[dr] \\
    & X \ar[d]_(.3){\delta_X} \ar@{->>}[rr] & & X +_{(f,g)} Y \ar[dll]^{\delta_X+_{\delta}\delta_Y} \\
    & B
} $$
\end{Proposition}

\begin{proof}
Thanks to Lemma 1.1 in \cite{BG}, both in \C\ and in \dense, a
commutative square of regular epimorphisms is a pushout if and
only if it induces a regular epimorphic restriction to the
kernels. Since the kernels in \dense\ are computed as in \C, the
pushout in \dense\ coincides with the pushout of the corresponding
arrows in \C.
\end{proof}

\dense\ being exact and sequentiable, the following proposition can be proved likewise the analogue one in the semi-abelian context (see Corollary 4.3.15 in \cite{Borceux-Bourn}).

\begin{Proposition} \label{prop:join.norm}
In \dense, the join of two kernels is the kernel of the diagonal of the pushout of the corresponding cokernels.
\end{Proposition}

\begin{Corollary} \label{cor:joinPX.ker}
The join of two kernels in \dense\ can be obtained by means of the join of the corresponding arrows in \C.
\end{Corollary}

Let us consider now the coproduct construction. The coproduct of two internal reflexive graphs $(X_1,d,c,e)$ and $(Y_1,d',c',e')$ in $\RG_B(\C)$ is given by the pushout of the corresponding initial maps:
$$ \xymatrix@=4ex{
    & 0 \ar[dl] \ar[rr] \ar@{ |>->}[dd] & & X \ar[dl]^f \ar@{ |>->}[dd] \\
    Y \ar@{ |>->}[dd] \ar[rr]^(.6)g & & X \plus Y \ar@{ |>->}[dd] \\
    & B \ar[dl]_{e'} \ar[rr]^(.3)e \ar@<-1ex>[dd]_(.3)1 \ar@<1ex>[dd]^(.3)1
        & & X_1 \ar[dl] \ar@<-1ex>[dd]_{c} \ar@<1ex>[dd]^{d} \\
    Y_1 \ar@<1ex>[dd]^{d'} \ar@<-1ex>[dd]_{c'} \ar[rr]
        & & X_1+_{(e,e')}Y_1 \ar@<1ex>[dd]^(.3){[d,d']} \ar@<-1ex>[dd]_(.3){[c,c']} \\
    & B \ar[uu]|(.7)1 \ar[rr]_(.3)1 \ar[dl]^(.4)1 & & B \ar[uu]|{e} \ar[dl]^1 \\
    B \ar[uu]|{e'} \ar[rr]_1 & & B \ar[uu]
} $$
But the induced square on the kernels of the domain projections is not in general a pushout in \C, so that $X \plus Y$ needs not coincide with the coproduct $X+Y$ in \C\ (see Proposition 6.2 in \cite{Gray12}). It does happen when the base category is \LACC\ (i.e.\ locally algebraically cartesian closed, see \cite{Gray12}), which implies that the kernel functor
$$ \Ker_B \colon \Pt_B(\C) \to \C $$
preserves colimits. However, while groups and Lie algebras are
examples of \LACC\ categories, many other important algebraic
varieties are not (e.g.\ the category of rings, as shown in
\cite{Gray12}). In Section \ref{sec:alg.coh}, we will use the
weaker condition \CS\ requiring the canonical arrow
$$ \sigma \colon X+Y \to X \plus Y $$
to be a regular epimorphism. In a semi-abelian context, this condition is equivalent to \emph{algebraic coherence} in the sense of \cite{CGrayVdL2} (see Proposition \ref{prop:ac.cs} below).

\begin{Remark} \label{rem:sigma.epi}
Notice that in any case the  comparison arrow $\sigma$ is  cancellable on the right with respect to morphisms in \dense, i.e.\ if $f$ and $g$ in \dense\ are such that $f \cdot \sigma = g\cdot \sigma$, then $f=g$.
\end{Remark}


\section{Peiffer product and Peiffer commutator} \label{sec:Peiffer}

\subsection{Definitions and properties} \label{sec:definitions}

\begin{Lemma} \label{lemma:pb.rg}
Let $(\delta\colon X \to B, \xi)$ be a pre-crossed module and $f \colon A \to B$ a morphism in \C. The pullback of $\delta$ along $f$ is endowed with a pre-crossed module structure.
\end{Lemma}

\begin{proof}
This is a pre-crossed module version of a standard fact about internal reflexive graphs and fully faithful morphisms between them (see \cite{AMMV} for the crossed module case). We just recall here the construction of the induced action: it is the (unique) arrow $\overline{\xi}$ making the diagram below commute.
$$ \xymatrix{
    A \flat (A \times_{(f,\delta)} X) \ar[dr]^-{\overline{\xi}} \ar[rr]^-{f \flat \overline{f}}
        \ar[dd]_{1 \flat \overline{\delta}} & & B \flat X \ar[d]^{\xi} \\
    & A \times_{(f,\delta)} X \ar[d]_{\overline{\delta}} \ar[r]^-{\overline{f}} & X \ar[d]^{\delta} \\
    A \flat A \ar[r]_-{\chi} & A \ar[r]_-f & B
} $$
\end{proof}

Let $(\delta_X \colon X \to B, \xi_X)$ and $(\delta_Y \colon Y \to B, \xi_Y)$ be pre-crossed modules in \C. Then $X$ and $Y$ act on each other by means of the following actions:
$$ \begin{array}{rl}
    \delta^*_Y\xi_X \colon & \!\!\!\!\xymatrix{Y \flat X \ar[r]^-{\delta_Y\flat 1}
        & B \flat X \ar[r]^-{\xi_X} & X} \\
    \delta^*_X\xi_Y \colon & \!\!\!\!\xymatrix{X \flat Y \ar[r]^-{\delta_X\flat 1}
        & B \flat Y \ar[r]^-{\xi_Y} & Y}
\end{array} $$

\begin{Proposition} \label{prop:semidir.PX}
There exists a (unique) arrow $[\delta_X,\delta_Y\rangle \colon X \rtimes Y \to B$, making the following diagram commute
$$ \xymatrix@=7ex{
    X \ar[r]^-{j_X} \ar[dr]_{\delta_X} & X \rtimes Y \ar@{-->}[d]^(.35){[\delta_X,\delta_Y\rangle}
        & Y \ar[l]_-{i_Y} \ar[dl]^{\delta_Y} \\
    & B
} $$
and it is endowed with a pre-crossed module structure, such that $j_X$ and $i_Y$ are morphisms in \dense.
\end{Proposition}

Notice that, in the case when $(Y,\delta_Y)=(B,1_B)$, we recover the arrow $[\delta_X,1\rangle\colon X \rtimes B \to B$ introduced in Section \ref{sec:PXMod}.

\begin{proof}
First of all we have to show that $[\delta_X,\delta_Y\rangle$ exists. This is true (and the arrow is unique) if and only if the following diagram commutes:
$$ \xymatrix@=7ex{
    Y \flat X \ar[r]^-{\kappa_{Y,X}} \ar[d]_{\delta^*_Y\xi_X} & X+Y \ar[d]^{[\delta_X,\delta_Y]} \\
    X \ar[r]_{\delta_X} & B
} $$
The latter depends on the fact that $\delta_X$ is a pre-crossed module.

Now we prove that there is an action of $B$ on $X \rtimes Y$. By definition of $\delta^*_Y\xi_X$, we can form the following pullback:
\begin{equation} \label{diag:delta.star.xi}
\begin{aligned}
\xymatrix{
    X \rtimes Y \ar[d]_{1 \rtimes \delta_Y} \ar[r]^-{p_Y} & Y \ar[d]^{\delta_Y} \\
    X \rtimes B \ar[r]_-{p_B} & B
}
\end{aligned}
\end{equation}
By Lemma \ref{lemma:pb.rg}, the arrow $1 \rtimes \delta_Y$ is a pre-crossed module, with the corresponding action $\overline{\xi_Y}$. Then $B$ acts on $X \rtimes Y$ with the action $i_B^*\overline{\xi_Y}$ defined by the following composition:
$$ \xymatrix{
    B \flat (X \rtimes Y) \ar[dr]_{i_B^*\overline{\xi_Y}} \ar[r]^-{i_B \flat 1}
        & (X \rtimes B) \flat (X \rtimes Y) \ar[d]^{\overline{\xi_Y}} \\
    & X \rtimes Y
} $$
Finally we prove that $([\delta_X,\delta_Y\rangle,i_B^*\overline{\xi_Y})$ is a pre-crossed module. It suffices to show that the following diagram commutes:
$$ \xymatrix{
    B\flat (X \rtimes Y) \ar[d]^-{i_B \flat 1} \ar@{}[drr]|{(a)}
        \ar[rrr]^-{1 \flat [\delta_X,\delta_Y\rangle} & & & B \flat B \ar[ddd]^{\chi} \\
    (X \rtimes B) \flat (X \rtimes Y) \ar[dd]^{\overline{\xi_Y}}
        \ar[rr]_{1 \flat (1 \rtimes \delta_Y)}
        & & (X \rtimes B) \flat (X \rtimes B) \ar[d]^{\chi}
        \ar[ur]^{[\delta_X,1\rangle \flat [\delta_X,1\rangle\quad} \ar@{}[dr]|{(c)} \\
    \ar@{}[rr]|{(b)} & & (X \rtimes B) \ar[dr]^{[\delta_X,1\rangle} & \\
    X \rtimes Y \ar[rrr]_{[\delta_X,\delta_Y\rangle} \ar[urr]_{1 \rtimes \delta_Y} & & & B
} $$
The commutativity of the bottom triangle follows from the uniqueness of $[\delta_X,\delta_Y\rangle$, and consequently $(a)$ commutes by functoriality of $-\flat-$. $(c)$ commutes because every morphism is equivariant with respect to the conjugation actions of its domain and codomain; the commutativity of $(b)$ depends on the fact that $1 \rtimes \delta_Y$ is a pre-crossed module.

In order to prove that $j_X \colon X \to X \rtimes Y$ is a morphism in \dense, we have to show that the following diagram commutes:
$$ \xymatrix{
    B \flat X \ar[rr]^{\xi_X} \ar[d]_{1 \flat j_X} & & X \ar[d]^{j_X} \\
    B \flat (X \rtimes Y) \ar[r]_-{i_B \flat 1}
        & (X \rtimes B) \flat (X \rtimes Y) \ar[r]_-{\overline{\xi_Y}}
        & X \rtimes Y
} $$
It suffices to compose with the jointly monic projections $(p_Y,1_X \rtimes \delta_Y)$ of the pullback (\ref{diag:delta.star.xi}). Indeed, $p_Y\cdot \overline{\xi_Y}=\xi_Y\cdot (p_B \flat p_Y)$ by definition of $\overline{\xi_Y}$, and then:
$$ p_Y \cdot \overline{\xi_Y} \cdot (i_B \flat 1) \cdot (1 \flat j_X) = \xi_Y \cdot (p_B \flat p_Y) \cdot (i_B \flat j_X) = \xi_Y \cdot (1 \flat 0) = 0 = p_Y \cdot j_X \cdot \xi_X \,. $$
On the other hand, by the commutativity of the square $(b)$ above, $(1 \rtimes \delta_Y)\cdot\overline{\xi_Y}=\chi\cdot(1\flat (1\rtimes \delta_Y))$, and then:
\begin{multline*}
    (1 \rtimes \delta_Y) \cdot \overline{\xi_Y} \cdot (i_B\flat 1) \cdot (1\flat j_X) = \chi \cdot
        (1\flat (1\rtimes \delta_Y)) \cdot (i_B\flat j_X) = \\
    = \chi \cdot (i_B \flat j_X) = j_X \cdot \xi_X = (1 \rtimes \delta_Y) \cdot j_X \cdot \xi_X \,,
\end{multline*}
where the last but one equality holds by definition of $X \rtimes B$ (in other words, the conjugation action of $B$ as subobject of $X \rtimes B$ on the normal subobject $X$ coincides with the action $\xi_X$ defining the semidirect product).

To prove that $i_Y \colon Y \to X \rtimes Y$ is also a morphism in \dense, we have to show that the following diagram commutes:
$$ \xymatrix{
    B \flat Y \ar[rr]^{\xi_Y} \ar[d]_{1 \flat i_Y} & & Y \ar[d]^{i_Y} \\
    B \flat (X \rtimes Y) \ar[r]_-{i_B \flat 1}
        & (X \rtimes B) \flat (X \rtimes Y) \ar[r]_-{\overline{\xi_Y}}
        & X \rtimes Y
} $$
As before, it suffices to compose with the projections of the pullback (\ref{diag:delta.star.xi}). Indeed:
$$ p_Y \cdot \overline{\xi_Y} \cdot(i_B \flat 1) \cdot(1 \flat i_Y) = \xi_Y \cdot (p_B \flat p_Y) \cdot (i_B \flat i_Y) = \xi_Y\cdot(1 \flat 1) = \xi_Y = p_Y \cdot i_Y \cdot \xi_Y \,, $$
and
\begin{multline*}
    (1 \rtimes \delta_Y) \cdot \overline{\xi_Y} \cdot (i_B\flat 1) \cdot (1\flat i_Y) = \chi \cdot
        (1\flat (1\rtimes \delta_Y))\cdot(i_B\flat i_Y) = \chi\cdot(i_B \flat (i_B\delta_Y)) = \\
    = \chi \cdot (i_B \flat i_B) \cdot (1 \flat \delta_Y) = i_B \cdot \chi \cdot (1 \flat \delta_Y)
        = i_B \cdot \delta_Y \cdot \xi_Y = (1 \rtimes \delta_Y) \cdot i_Y \cdot \xi_Y \,,
\end{multline*}
where in the last line we use the fact that $\delta_Y$ is a pre-crossed module.
\end{proof}

As a consequence, we get a morphism $[j_X,i_Y]_\PX\colon X \plus Y \to X \rtimes Y$ in \dense. We denote by $X \twcpxl Y$ the kernel of such an arrow in \C. As seen in Section \ref{sec:limits}, $0 \colon X \twcpxl Y \to B$ is the kernel in \dense\ of $[j_X,i_Y]_\PX$. In a symmetric way, we obtain $X \twcpxr Y$ as the kernel of $[i_X,j_Y]_\PX \colon X \plus Y \to Y \rtimes X$.

\begin{Definition} \label{def:peiff.prod}
We denote by $X \twcpx Y := (X \twcpxl Y) \vee_{\PX} (X \twcpxr Y)$ the domain of the pre-crossed module which is the join in \dense\ of these two normal subobjects of $X \plus Y$. By Proposition \ref{prop:join.norm}, this join is the kernel of the diagonal of the following pushout in \dense:
$$ \xymatrix@C=9ex{
    X \plus Y \ar[r]^-{[j_X,i_Y]_\PX} \ar[d]_{[i_X,j_Y]_\PX} & X \rtimes Y \ar[d] \\
    Y \rtimes X \ar[r] & X \farf Y
} $$
where $X \farf Y$ is what we call the (internal) \emph{Peiffer product} of $\delta_X$ and $\delta_Y$, with $\delta_{X \farf Y}$ the corresponding pre-crossed module over $B$ (see \cite{GilHig} for the original definition of Peiffer product of groups).
\end{Definition}

In the following, we will denote $l_X$ and $l_Y$ the canonical morphisms from $X$ and $Y$, respectively, to $X \farf Y$, obtained by composition of the canonical injections in $X \plus Y$ with the regular epimorphism $X \plus Y \to X \farf Y$. It follows that $(l_X,l_Y)$ is a jointly strongly epimorphic pair in \dense.

\begin{Remark}
Given two trivial pre-crossed $B$-modules associated with $X$ and $Y$ (i.e.\ $\delta=0$ with $B$ acting trivially), we have $X \rtimes Y \cong Y \rtimes X \cong X\times Y$ and consequently $X \farf Y \cong X\times Y$.
\end{Remark}

\begin{Lemma} \label{lem:pushout.C}
Let $f\colon X \to X'$ and $g\colon Y \to Y'$ be regular epimorphisms in \dense. Then the following is a pushout in \C:
$$ \xymatrix{
    X+Y \ar@{->>}[r]^-{[j,i]} \ar@{->>}[d]_{f+g} & X \rtimes Y \ar@{->>}[d]^{f \rtimes g} \\
    X'+Y' \ar@{->>}[r]_-{[j,i]} & X' \rtimes Y'
} $$
\end{Lemma}

\begin{proof}
Let us consider the following commutative cube:
$$ \xymatrix@=4ex{
    & Y\flat X \ar[dl]_{\kappa_{Y,X}} \ar@{->>}[dd]_(.7){g\flat f} \ar[rr]^{\delta_Y^*\xi_X}
        & & X \ar@{->>}[dd]^f \ar[dl]^j \\
    X+Y \ar@{->>}[dd]_{f+g} \ar@{->>}[rr]^(.7){[j,i]} & & X \rtimes Y \ar@{->>}[dd]^(.3){f \rtimes g} \\
    & Y' \flat X' \ar[dl]^(.4){\kappa_{Y',X'}} \ar[rr]_(.35){\delta_{Y'}^*\xi_{X'}} & & X' \ar[dl]^{j} \\
    X'+Y' \ar@{->>}[rr]_-{[j,i]} & & X'\rtimes Y'
} $$
The upper and lower squares are pushouts by definition of semidirect product. $g\flat f$ being a regular epimorphism (since $-\flat -$ preserves them) and ${\delta_{Y'}^*\xi_{X'}}$ a split epimorphism,  the square on the rear and hence its composite with the lower square are pushouts. By cancellation, it follows that the square on the front is a pushout.
\end{proof}

\begin{Lemma} \label{lem:pushout.dense}
Let $f\colon X \to X'$ and $g\colon Y \to Y'$ be regular epimorphisms in \dense. Then the following is a pushout in \dense:
$$ \xymatrix{
    X\plus Y \ar@{->>}[r]^-{[j,i]_\PX} \ar@{->>}[d]_{f\plus g} & X \rtimes Y \ar@{->>}[d]^{f \rtimes g} \\
    X'\plus Y' \ar@{->>}[r]_-{[j,i]_\PX} & X' \rtimes Y'
} $$
\end{Lemma}

\begin{proof}
Let us consider the diagram of solid arrows below, where $u$ and $v$ are morphisms in \dense\ such that $u\cdot(f \plus g)=v\cdot[j,i]_\PX$:
$$ \xymatrix@C=7ex{
    X+Y \ar@{}[dr]|{(a)} \ar[r]^-{\sigma} \ar@{->>}[d]_{f+g}
        & X \plus Y \ar@{}[dr]|{(b)} \ar@{->>}[d]_{f \plus g} \ar@{->>}[r]^-{[j,i]_\PX}
        & X \rtimes Y \ar@{->>}[d]_{f \rtimes g} \ar@/^/[ddr]^v \\
    X'+Y' \ar[r]_-{\sigma} & X' \plus Y' \ar@/_/[drr]_u \ar@{->>}[r]^-{[j,i]_\PX}
        & X' \rtimes Y' \ar@{-->}[dr]_t \\
    & & & \bullet
} $$
By Lemma \ref{lem:pushout.C}, the rectangle $(a)+(b)$ is a pushout in \C, hence there exists a unique $t$ in \C\ such that $t\cdot (f\rtimes g)=v$ and $t\cdot[j,i]_\PX\cdot\sigma=u\cdot\sigma$. Moreover, $t$ is a morphism in \dense\ by Lemma \ref{lem:3outof2}, since $v$ and the regular epimorphism $(f\rtimes g)$ are. Then the equality $t\cdot[j,i]_\PX=u$ follows by applying Remark \ref{rem:sigma.epi}.
\end{proof}

\begin{Proposition} \label{prop:epireg}
Let $f\colon X \to X'$ and $g\colon Y \to Y'$ be regular epimorphisms in \dense. Then the induced arrows
\begin{gather*}
    f \farf g \colon X \farf Y \to X' \farf Y' \\
    f \twcpx g \colon X \twcpx Y \to X' \twcpx Y'
\end{gather*}
are also regular epimorphisms.
\end{Proposition}

\begin{proof}
The arrow $f \farf g$ is obviously a regular epimorphism being the last part of a composite of regular epimorphisms. By Lemma \ref{lem:pushout.dense} and a trivial argument on composition and cancellation of pushouts the following is a pushout in \dense:
$$ \xymatrix@C=9ex{
    X \plus Y \ar@{->>}[d]_{f \plus g} \ar@{->>}[r]^-{[l_X,l_Y]_\PX} & X \farf Y \ar@{->>}[d]^{f \farf g} \\
    X' \plus Y' \ar@{->>}[r]_-{[l_{X'},l_{Y'}]_\PX} & X' \farf Y'
} $$
As a consequence, the restriction to the kernels of the horizontal arrows, i.e.\ the morphism $f \twcpx g$, is a regular epimorphism. Notice that all the pushouts above are also pushouts in \C\ as observed in Section \ref{sec:colimits}.
\end{proof}

\begin{Definition} \label{def:Peiff.comm}
Let $(\delta_X \colon X \to B, \xi_X)$ and $(\delta_Y \colon Y \to B, \xi_Y)$ be subobjects of $(\delta \colon A \to B, \xi)$ in \dense:
\begin{equation} \label{diag:mn}
\begin{aligned}
\xymatrix@=7ex{
    X \ar@{ >->}[r]^-{m} \ar[dr]_{\delta_X} & A \ar[d]^{\delta} & Y \ar@{ >->}[l]_-{n} \ar[dl]^{\delta_Y} \\
    & B
}
\end{aligned}
\end{equation}
The \emph{Peiffer commutator} $\langle X,Y \rangle$ is given by the regular image of $X \twcpx Y$ through the arrow $[m,n]_\PX\colon X \plus Y \to A$:
\begin{equation} \label{diag:peiff.comm}
\begin{aligned}
\xymatrix{
    X \twcpx Y \ar@{ |>->}[d] \ar@{->>}[r] & \langle X,Y \rangle \ar@{ >->}[d] \\
    X \plus Y \ar[r]_-{[m,n]_\PX} & A
}
\end{aligned}
\end{equation}
\end{Definition}

\begin{Remark}
Since the (regular epi, mono) factorization in \dense\ is the same as in\, \C, $\langle X,Y \rangle$ turns out to be a subobject of $A$ in \dense. Its corresponding arrow on $B$ is $0 \colon \langle X,Y \rangle \to B$, since it is the image of a kernel. Hence, it is possible to compute the cokernel in \dense\ of the inclusion of $\langle X,Y \rangle$ in $A$.
\end{Remark}

\begin{Remark}
$X \twcpx Y$ is the Peiffer commutator of $X$ and $Y$ in $X \plus Y$.
\end{Remark}

From the definition of Peiffer commutator and being the Peiffer product a cokernel, the next result follows.

\begin{Proposition} \label{prop:tw.coop}
Let $X$ and $Y$ be subobject of $A$ in \dense\ as in diagram (\ref{diag:mn}). The following are equivalent:
\begin{enumerate}
    \item $\langle X,Y \rangle=0$;
    \item there exists a (necessarily unique) morphism $\varphi$ making the following diagram commute:
$$ \xymatrix@=7ex{
    X \ar[r]^-{l_X} \ar@{ >->}[dr]_{m} & X \farf Y \ar@{-->}[d]^{\varphi} & Y \ar[l]_-{l_Y} \ar@{ >->}[dl]^{n} \\
    & A
} $$
\end{enumerate}
\end{Proposition}


\begin{Remark}
In case of trivial pre-crossed modules associated with $X$ and $Y$, as already observed \mbox{$X \farf Y \cong X \times Y$}, $l_X=\langle 1,0 \rangle$ and $l_Y=\langle 0,1 \rangle$, and $\varphi$ is nothing but the cooperator of $f$ and $g$ in the sense of \cite{Borceux-Bourn}. As a consequence, in this case, the normal closure of the Peiffer commutator coincides with the Huq commutator.
\end{Remark}

In the category of groups, the Peiffer commutator of $X$ and $Y$ defined above is the normal closure in $X \vee Y$ of the Peiffer commutator defined by Conduch\'e and Ellis in \cite{CE}. Notice that in this case $X \plus Y = X + Y$. This is not necessarily true in a general semi-abelian category, hence the computation of the Peiffer commutator may not be  easy. As we will see in Section \ref{sec:alg.coh}, when \C\ is algebraically coherent, the construction of the Peiffer commutator can be performed entirely in \C, avoiding the use of the coproduct in \dense.

\begin{Proposition} \label{prop:peiff.fact}
The Peiffer commutator preserves the (regular epi, mono) factorization in \dense. Namely, given the following commutative diagram in \dense:
$$ \xymatrix{
    X \ar@{ >->}[r]^{m} \ar[d]_f & A \ar[d]_h & Y \ar@{ >->}[l]_{n} \ar[d]^g \\
    X' \ar@{ >->}[r]_{m'} & A' & Y' \ar@{ >->}[l]^{n'}
} $$
where $m$, $n$, $m'$ and $n'$ are monomorphisms, we have the following factorization of the induced arrow between the Peiffer commutators:
$$ \langle f,g \rangle \colon \xymatrix{
    \langle X,Y \rangle \ar@{->>}[r] & \langle f(X),g(Y) \rangle \ar@{ >->}[r] & \langle X',Y' \rangle
} $$
\end{Proposition}

\begin{proof}
By the properties of the factorization, it suffices to show that:
\begin{enumerate}
    \item $\langle f,g \rangle$ is a regular epimorphism whenever $f$, $g$ and $h$ are;
    \item $\langle f,g \rangle$ is a monomorphism whenever $f$, $g$ and $h$ are.
\end{enumerate}
The second assertion is trivial, while the first one follows from Proposition \ref{prop:epireg}, since the diagonal of the following commutative square is a regular epimorphism:
$$ \xymatrix{
    X \twcpx Y \ar@{->>}[d]_{f \twcpx g} \ar@{->>}[r] & \langle X,Y \rangle \ar[d]^{\langle f,g \rangle} \\
    X' \twcpx Y'  \ar@{->>}[r] & \langle X',Y' \rangle
} $$
\end{proof}

\begin{Corollary} \label{cor:monotone}
The Peiffer commutator is monotone: if $X\leq X'$ and $Y \leq Y'$ are pre-crossed submodules of a given pre-crossed module $A$, then $\langle X,Y \rangle \leq \langle X',Y' \rangle$.
\end{Corollary}

\begin{Corollary} \label{cor:image=0}
If $X$ and $Y$ are pre-crossed submodules of a given pre-crossed module $A$ and $q$ denotes the cokernel of the inclusion of $\langle X,Y \rangle$ in $A$, then $\langle q(X),q(Y) \rangle=0$.
\end{Corollary}

\subsection{Reflection onto crossed modules}\label{sec:reflection}

We are going to show how the Peiffer commutator may allow to describe directly the reflection
$$ \xymatrix{
    \dense \ar@<1ex>[r]^-I \ar@{}[r]|-\bot & \XMod_B(\C) \ar@<1ex>[l]^-H
} $$
where $\XMod_B(\C)$ stands for the subcategory of internal crossed modules of codomain $B$, introduced by Janelidze in \cite{Janelidze}. We recall from \cite{MFVdL} that, when the semi-abelian category \C\ satisfies the condition \SH, an internal pre-crossed module $(\partial\colon A \to B,\xi)$ is a crossed module if and only if the following diagram commutes (Peiffer condition):
$$ \xymatrix{
    A \flat A \ar[d]_{\chi} \ar[r]^-{\partial \flat 1} & B \flat A \ar[d]^{\xi} \\
    A \ar[r]_1 & A
} $$

As a corollary of Proposition \ref{prop:tw.coop} we get the following lemma.

\begin{Lemma} \label{lem:xmod=>peiff=0}
Given a crossed module $(\partial\colon A \to B,\xi)$, we have $\langle A,A \rangle=0$. Moreover, if $X$ and $Y$ are pre-crossed submodules of $A$, then $\langle X,Y \rangle=0$.
\end{Lemma}

\begin{proof}
Since $\partial$ is a crossed module, it satisfies the Peiffer condition, which is equivalent to the commutativity of the following diagram
$$ \xymatrix{
    A \flat A \ar[d]_{\partial^*\xi} \ar[r] & A+A \ar[d]^{[1,1]} \\
    A \ar[r]_1 & A
} $$
which, in turn, is equivalent to the existence of the (unique) arrow $[1,1\rangle$ making the following diagram commute:
$$ \xymatrix{
    A \ar[r]^-{j} \ar[dr]_{1} & A \rtimes A \ar[d]^(.4){[1,1\rangle}
        & A \ar[l]_-{i} \ar[dl]^{1} \\
    & A
} $$
By Proposition \ref{prop:semidir.PX}, we know that $j$ and $i$ are morphisms in \dense\ of codomain $[\partial,\partial\rangle\colon A \rtimes A \to B$, then there exists in \dense\ the canonical arrow $[j,i]_\PX\colon A \plus A \to A \rtimes A$, which is a regular epimorphism since $(j,i)$ is a jointly strongly epimorphic pair in \C\ and then in \dense. Since both $[j,i]_\PX$ and the composite $[1,1\rangle[j,i]_\PX=[1,1]_\PX$ are morphisms in \dense\ and $[j,i]_\PX$ is a regular epimorphism, by Lemma \ref{lem:3outof2} $[1,1\rangle$ is also a morphism in \dense:
$$ \xymatrix{
    A \rtimes A \ar[dr]_{[\partial,\partial\rangle} \ar[rr]^-{[1,1\rangle} & & A \ar[dl]^{\partial} \\
    & B
} $$
Moreover, $[1,1\rangle[j,i]_\PX=[1,1\rangle[i,j]_\PX=[1,1]_\PX$. From the definition of $A \farf A$ it follows that there exists a (unique) arrow $\varphi$ making the following diagram commute:
$$ \xymatrix@C=10ex{
    A \plus A \ar[r]^{[j,i]_\PX} \ar[d]_{[i,j]_\PX} & A \rtimes A \ar[d] \ar@/^/[ddr]^{[1,1\rangle} \\
    A \rtimes A \ar[r] \ar@/_/[drr]_{[1,1\rangle} & A \farf A \ar[dr]_\varphi \\
    & & A
} $$
By Proposition \ref{prop:tw.coop} this means that $\langle A,A \rangle=0$.

The second statement follows from Corollary \ref{cor:monotone}.
\end{proof}

In fact, when the condition \SH\ holds (as in the case of strongly protomodular categories and in particular for algebraically coherent categories), the property $\langle A,A \rangle=0$ characterizes crossed modules among pre-crossed modules.

\begin{Proposition} \label{prop:char.xmod}
In a semi-abelian category \C\ satisfying \SH, a pre-crossed module $(\partial\colon A \to B,\xi)$ is a crossed module if and only if $\langle A,A \rangle=0$.
\end{Proposition}

\begin{proof}
The condition $\langle A,A \rangle=0$ is necessary by Lemma \ref{lem:xmod=>peiff=0}.

Suppose now that $\langle A,A \rangle=0$. Then, by Proposition \ref{prop:tw.coop}, there exists a unique arrow $\varphi$ making the following diagram commute:
$$ \xymatrix{
    A \ar[r]^-{l_1} \ar[dr]_{1} & A \farf A \ar@{-->}[d]^{\varphi} & A \ar[l]_-{l_2} \ar[dl]^{1} \\
    & A
} $$
By composition with the canonical arrow $A \rtimes A \to A \farf A$, we get the arrow $[1,1\rangle$ below:
$$ \xymatrix{
    A \ar[r]^-{j} \ar[dr]_{1} & A \rtimes A \ar[d]^(.4){[1,1\rangle}
        & A \ar[l]_-{i} \ar[dl]^{1} \\
    & A
} $$
As observed in the proof of Lemma \ref{lem:xmod=>peiff=0}, the existence of $[1,1\rangle$ is equivalent to the Peiffer condition, which, under the \SH\ condition, implies that $(\partial\colon A \to B,\xi)$ is a crossed module.
\end{proof}

Furthermore the reflection $\dense \to \XMod_B(\C)$ can be obtained performing the quotient on the Peiffer commutator. More precisely, the following theorem holds.

\begin{Theorem}
Let $(\delta\colon X \to B,\xi)$ be an internal pre-crossed module in a semi-abelian category \C\ satisfying \SH. Then $HI(\delta)$ is obtained by the following cokernel in \dense:
$$ \xymatrix{
    \langle X,X \rangle \ar@{ |>->}[r] \ar[dr]_0 & X \ar[d]_\delta \ar@{-|>}[r]^-{\eta}
        & \dfrac{X}{\langle X,X \rangle} \ar[dl]^{HI(\delta)} \\
    & B
} $$
\end{Theorem}

\begin{proof}
The regular epimorphism $\eta$ arises also as the right vertical map in the following pushout in \dense:
\begin{equation} \label{diag:reflection}
\begin{aligned}
\xymatrix{
    X \plus X \ar@{-|>}[r]^-{[1,1]_\PX} \ar@{-|>}[d]_\Sigma & X \ar@{-|>}[d]^-\eta \\
    X \farf X \ar@{-|>}[r]_-q & \dfrac{X}{\langle X,X \rangle}
}
\end{aligned}
\end{equation}
(paste this diagram with diagram (\ref{diag:peiff.comm})).

The fact that $HI(\delta)$ is a crossed module follows from Corollary \ref{cor:image=0} and Proposition \ref{prop:char.xmod}.

On the other hand, given any morphism $f$ in \dense:
$$ \xymatrix{
    X \ar[rr]^f \ar[dr]_\delta & & A \ar[dl]^\partial \\
    & B
} $$
where the codomain is a crossed module, then by Lemma \ref{lem:xmod=>peiff=0} we have that \mbox{$\langle f(X),f(X) \rangle=0$} and by Proposition \ref{prop:tw.coop} this yields an arrow $\varphi \colon X \farf X \to A$ making the following diagram commute:
$$ \xymatrix{
    X \ar[r]^-{l_1} \ar[dr]_{f} & X \farf X \ar@{-->}[d]^{\varphi} & X \ar[l]_-{l_2} \ar[dl]^{f} \\
    & A
} $$
Hence $\varphi\cdot\Sigma = f\cdot[1,1]_\PX$ and the universal property of $HI(\delta)$ follows from the one of the pushout (\ref{diag:reflection}).
\end{proof}


\section{The algebraically coherent case} \label{sec:alg.coh}

We consider from now on the following condition on the base
category \C:

\begin{itemize}
    \item[(CS)] For any $(\delta_X\colon X \to B)$ and $(\delta_Y\colon Y \to B)$ in \dense, the comparison arrow
        $$ \sigma \colon X+Y \to X \plus Y $$
is a regular epimorphism.
\end{itemize}

It turns out (as Proposition \ref{prop:ac.cs} shows) that, in the semi-abelian context, this condition is equivalent to \emph{algebraic coherence}, independently introduced in \cite{CGrayVdL2}. A semi-abelian category is algebraically coherent when, for any $B$ in \C, the kernel functor
$$ \Ker \colon \Pt_B(\C) \to \C $$
preserves jointly strongly epimorphic pairs. This condition is fulfilled by a wide class of algebraic varieties such as the categories of groups, rings, Lie and Leibniz algebras, Poisson algebras and in general any \emph{category of interest} in the sense of Orzech \cite{Orzech}, as shown in \cite{CGrayVdL2}.

\begin{Proposition} \label{prop:ac.cs}
A semi-abelian category \C\ is algebraically coherent if and only if it satisfies the condition \CS.
\end{Proposition}

\begin{proof}
Given $X$ and $Y$ in \dense, it suffices to consider the coproduct of the corresponding objects in $\RG_B(\C)$ and restrict the canonical injections to the kernels of the domain projections. If \C\ is algebraically coherent, these restrictions are jointly strongly epimorphic or, equivalently, $\sigma$ is a regular epimorphism.

Conversely, it suffices to observe that $\Pt_B(\C)$ is isomorphic to the full subcategory of $\RG_B(\C)$ whose objects are those reflexive graphs where domain and codomain projections coincide, and this embedding preserves coproducts. Then, given a jointly strongly epimorphic pair $(f_1,g_1)$ in $\Pt_B(\C)$, whose restriction to kernels is $(f,g)$:
$$ \xymatrix{
    X \ar[r]^f \ar@{ |>->}[d] & A \ar@{ |>->}[d] & Y \ar[l]_g \ar@{ |>->}[d] \\
    X_1 \ar@<.5ex>[d]^{p} \ar[r]^{f_1} & A_1 \ar@<.5ex>[d]^{\alpha} & Y_1 \ar@<.5ex>[d]^{p'} \ar[l]_{g_1} \\
    B \ar@<.5ex>[u]^{s} \ar[r]_1 & B \ar@<.5ex>[u]^{\beta} \ar[r]_1 & B \ar@<.5ex>[u]^{s'}
} $$
just doubling the split epimorphisms, one gets a jointly strongly epimorphic pair in $\RG_B(\C)$, and, by equivalence, in \dense. Hence, the arrow $[f,g]_\PX \colon X \plus Y \to A$ is a regular epimorphism. If $\sigma$ is a regular epimorphism, by composition $[f,g] \colon X+Y \to A$ is also a regular epimorphism, proving that the kernel functor $\Ker \colon \Pt_B(\C) \to \C$ preserves jointly strongly epimorphic pairs as desired.
\end{proof}

We saw in Section \ref{sec:colimits} that the join of two kernels in \dense\ coincides with the join of the corresponding objects in \C. Under \CS, this property extends to arbitrary subobjects. In fact, we have more.

\begin{Proposition} \label{prop:joinPX}
A semi-abelian category \C\ satisfies the condition \CS\ if and only if the join of any pair of subobjects in \dense\ coincides with the join of the corresponding subobjects in \C:
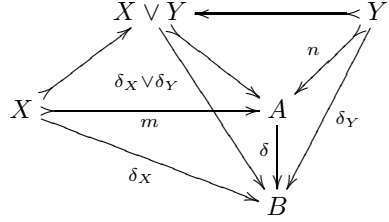
$$ \xymatrix{
    & X \vee Y \ar@{ >->}[dr] \ar[ddr]_(.3){\delta_X \vee \delta_Y}
        & & Y \ar@{ >->}[ll] \ar[ddl]^{\delta_Y} \ar@{ >->}[dl]_{n} \\
    X \ar@{ >->}[ur] \ar[drr]_{\delta_X} \ar@{ >->}[rr]_m & & A \ar[d]_(.4)\delta \\
    & & B
} $$
\end{Proposition}

\begin{proof}
The joins of $X$ and $Y$ in \C\ and in \dense\ are obtained by means of the (regular epi, mono) factorization of the arrows $[m,n] \colon X+Y \to A$ and $[m,n]_\PX \colon X \plus Y \to A$ respectively. If \CS\ holds, then the two factorizations yield the same subobject of $A$.

Conversely, let us consider $X$ and $Y$ as subobjects of $X \plus Y$ in \dense. If their join in \C\ coincides with $X \plus Y$, then $\sigma \colon X+Y \to X \plus Y$ is a regular epimorphism.
\end{proof}

\subsection{Construction of the Peiffer product and Peiffer commutator under \CS}

The condition \CS\ on \C\ says that $\sigma\colon X+Y \to X\plus Y$ is a regular epimorphism. Hence, considering the following commutative diagram where the two rows are short exact sequences:
$$ \xymatrix{
    X \twcl Y  \ar@{ |>->}[r] \ar@{->>}[d] & X+Y \ar@{->>}[r]^-{[j,i]} \ar@{->>}[d]_-{\sigma}
        &  X \rtimes Y \ar[d]^1 \\
    X \twcpxl Y \ar@{ |>->}[r] & X \plus Y \ar@{->>}[r]_-{[j,i]_{\PX}} & X \rtimes Y
} $$
the leftmost vertical arrow is a regular epimorphism, since the left hand square is a pullback. In a symmetric way, we obtain a canonical regular epimorphism by replacing $X \rtimes Y$ with $Y \rtimes X$ in the diagram above:
$$ \xymatrix{X \twcr Y \ar@{->>}[r] & X \twcpxr Y} $$
Let us denote $X \twc Y := (X \twcl Y) \vee (X \twcr Y)$ (computed
in $X+Y$). By Corollary \ref{cor:joinPX.ker}, it follows that also
$X \twcpx Y $ can be obtained as the join $ (X \twcpxl Y) \vee (X
\twcpxr Y)$ in \C. Then we have a regular epimorphism:
$$ \xymatrix{X \twc Y \ar@{->>}[r] & X \twcpx Y} $$

\begin{Proposition} \label{prop:peiff.prod.ac}
The Peiffer product $X \farf Y$, introduced in Definition
\ref{def:peiff.prod} as the quotient $\frac{X \plus Y}{X \twcpx
Y}$ in \dense, under \CS\ can be obtained as a quotient
$\frac{X+Y}{X \twc Y}$ of the coproduct in \C.
\end{Proposition}

\begin{proof}
Consider the following diagram, where the inner square is constructed as a pushout in \dense:
$$ \xymatrix@C=9ex{
    X+Y \ar@/^/[drr]^{[j,i]} \ar@/_/[ddr]_{[i,j]} \ar@{->>}[dr]^{\sigma} \\
    & X \plus Y \ar[r]_-{[j,i]_\PX} \ar[d]^{[i,j]_\PX} & X \rtimes Y \ar[d] \\
    & Y \rtimes X \ar[r] & X \farf Y
} $$
Since $\sigma$ is a regular epimorphism and the inner square is also a pushout in \C, so is the outer square. \end{proof}

\begin{Remark} \label{rem:w}
In this case, the Peiffer product $X \farf Y$ can also be computed
as the colimit of the following diagram:
$$ \xymatrix{
    & Y \flat X \ar[dl]_{\delta_Y^*\xi_X} \ar[dr]^{\kappa_{_{Y,X}}}
        & & X \flat Y \ar[dr]^{\delta_X^*\xi_Y} \ar[dl]_{\kappa_{_{X,Y}}} \\
    X & & X+Y & & Y
} $$
where $\kappa_{_{Y,X}}$ and $\kappa_{_{X,Y}}$ represent the formal conjugates in $X+Y$ of\, $Y$ on $X$ and of \,$X$ on $Y$. Indeed,
in order to obtain this colimit, we can take first the pushout of the span on the left  producing the semi-direct product $X \rtimes Y$ (see \cite{Janelidze}),
 i.e.\ the universal quotient of $X+Y$ where the action ${\delta_Y^*\xi_X}$ becomes a conjugation. In a symmetric way, we obtain $Y \rtimes X$ and  by a final pushout, we get the desired colimit, which coincides then with $X \farf Y$, thanks to Proposition \ref{prop:peiff.prod.ac}. As a consequence, $X \farf Y$ is the universal quotient of $X+Y$ where both the actions ${\delta_Y^*\xi_X}$ and ${\delta_X^*\xi_Y}$ become conjugations.
\end{Remark}

\begin{Proposition}
Let $X$ and $Y$ be pre-crossed submodules of $A$ as in diagram
(\ref{diag:mn}). Under \CS\ the Peiffer commutator $\langle X,Y
\rangle$ can be obtained as the regular image of $X \twc Y$
through the arrow $[m,n] \colon X + Y \to A$:
$$ \xymatrix@C=7ex{
    X \twc Y \ar@{ |>->}[d] \ar@{->>}[r] & X \twcpx Y \ar@{ |>->}[d] \ar@{->>}[r]
        & \langle X,Y \rangle \ar@{ >->}[d] \\
    X+Y \ar@{->>}[r]^-{\sigma} \ar@/_1pc/[rr]_-{[m,n]} & X \plus Y \ar[r]^-{[m,n]_\PX} & A
} $$
\end{Proposition}

\begin{Remark}
In the case $[m,n]$ is a regular epimorphism, then $\langle X,Y \rangle$ is normal in $A$ if considered in \C. But since algebraic coherence implies strong protomodularity (see \cite{CGrayVdL2}), thanks to Proposition \ref{prop:normal}, it becomes normal also in \dense.
\end{Remark}

\begin{Remark}
In case of trivial pre-crossed modules, $X \twc Y$ coincides with the canonical object $X\, {\Diamond} \,Y$ of \cite{MaMe10-2} and the Peiffer commutator $\langle X,Y \rangle$ coincides with the Higgins commutator $[X,Y]$ introduced in the case of $\Omega$-groups in \cite{Higgins} and in a categorical setting in \cite{MaMe10-2}.
\end{Remark}

\begin{Remark}
In the present setting the reflection of a pre-crossed module
$(\delta\colon X \to B,\xi)$ described in Section
\ref{sec:reflection} can be computed by means of the
following pushout in \C:
$$ \xymatrix{
    X + X \ar@{-|>}[r]^-{[1,1]} \ar@{-|>}[d] & X \ar@{-|>}[d]^-\eta \\
    X \farf X \ar@{-|>}[r]_q & \dfrac{X}{\langle X,X \rangle}
} $$
\end{Remark}

\subsection{Examples}

As an application of the constructions described above, we provide here some examples of explicit calculation of the Peiffer commutator for algebraically coherent varieties. In this context, following the description given in Remark \ref{rem:w}, we can interpret $\langle X,Y \rangle$ as the ideal generated in $X\vee Y$ by the Peiffer words, i.e.\ those elements of $A$ whose vanishing  makes the actions ${\delta_Y^*\xi_X}$ and ${\delta_X^*\xi_Y}$ become conjugations.

\begin{Example}
Let us consider first the case of (not necessarily unitary) rings. Following the notation of diagram (\ref{diag:mn}), let $(\delta\colon A \to B,\xi)$ be a pre-crossed  module in the category of rings. The action $\xi$ is given by the assignment of two bilinear maps:
$$ \begin{array}{l}
    B \times A \rightarrow A, \qquad (b,a) \mapsto b \cdot a\,, \\
    A \times B \rightarrow A, \qquad (a,b) \mapsto a \cdot b\,,
\end{array} $$
satisfying the following identities (for all $a,a'\in A$ and $b,b'\in B$):
$$ \begin{array}{ll}
    (bb') \cdot a = b \cdot (b' \cdot a)\,, & (b \cdot a) a' = b \cdot (aa')\,, \\
    (b \cdot a) \cdot b' = b \cdot (a \cdot b')\,, & (a \cdot b) a' = a (b \cdot a')\,, \\
    (a \cdot b) \cdot b' = a \cdot (bb')\,, & (aa') \cdot b = a (a' \cdot b) \,.
\end{array} $$
The pre-crossed module condition says that:
$$ \delta(b \cdot a) = b \cdot \delta(a) \qquad \mbox{and} \qquad \delta(a \cdot b) = \delta(a) \cdot b. $$
Given $X$ and $Y$ pre-crossed submodules of $A$ as in diagram (\ref{diag:mn}), their Peiffer commutator $\langle X, Y \rangle \leq A$ is the ideal of $X \vee Y$ generated by the following Peiffer words (for all $x\in X$ and $y\in Y$):
$$ xy-\delta_X(x) \cdot y\,,\quad xy-x \cdot \delta_Y(y)\,,\quad yx-\delta_Y(y)\cdot x\,,\quad yx-y \cdot \delta_X(x)\,. $$
\end{Example}

\begin{Example}
Consider now the category of Leibniz algebras over a fixed field. As above, $\delta$ is a pre-crossed module. Here, the action $\xi$ is a pair of bilinear maps:
$$ \begin{array}{l}
    B \times A \rightarrow A, \qquad (b,a) \mapsto [b , a]\,, \\
    A \times B \rightarrow A, \qquad (a,b) \mapsto [a,b]\,,
\end{array} $$
satisfying the following identities (for all $a,a'\in A$ and $b,b'\in B$):
$$ \begin{array}{ll}
    \left[[a,a'],b\right]=\left[[a,b],a'\right]+\left[a,[a',b]\right]
        & \left[[a,b],b'\right]=\left[[a,b'],b\right]+\left[a,[b,b']\right] \\
    \left[[a,b],a'\right]=\left[[a,a'],b\right]+\left[a,[b,a']\right]
        & \left[[b,a],b'\right]=\left[[b,b'],a\right]+\left[b,[a,b']\right] \\
    \left[[b,a],a'\right]=\left[[b,a'],a\right]+\left[b,[a,a']\right]
        & \left[[b,b'],a\right]=\left[[b,a],b'\right]+\left[b,[b',a]\right]
\end{array} $$
and the pre-crossed module condition says that:
$$ \delta([b, a]) = [b, \delta(a)] \qquad \mbox{and} \qquad \delta([a, b]) = [\delta(a), b] \,. $$
Given $X$ and $Y$ pre-crossed submodules of $A$ as in diagram (\ref{diag:mn}), their Peiffer commutator $\langle X, Y \rangle \leq A$ is the ideal of $X \vee Y$ generated by the following Peiffer words (for all $x\in X$ and $y\in Y$):
$$ [x,y]-[\delta_X(x),y]\,,\quad [x,y]-[x,\delta_Y(y)]\,,\quad [y,x]-[\delta_Y(y),x]\,,\quad [y,x]-[y,\delta_X(x)]\,. $$
\end{Example}


\section{Coproduct of crossed modules} \label{sec:coproduct}

Brown showed in \cite{BrownCCM} that the Peiffer product of two crossed $B$-modules of groups represents their coproduct in $\XMod_B(\Gp)$. We show here that an internal version of this result holds in any algebraically coherent semi-abelian category satisfying the following condition.

\begin{itemize}
    \item [(UA)] Given a jointly strongly epimorphic cospan $\xymatrix{A \ar[r]^f & B & C \ar[l]_g}$ in \C, then for any 4-tuple $(\xi_1,\xi_2,\xi_3,\xi_4)$ of actions on a fixed object $X$ making the following diagram commute
\begin{equation} \label{diag:C}
\begin{aligned}
\xymatrix{
    A \flat X \ar[dr]_{\xi_1} \ar[r]^{f\flat 1} & B\flat X \ar@<-.5ex>[d]_(.4){\xi_3} \ar@<.5ex>[d]^(.4){\xi_4}
        & C \flat X \ar[dl]^{\xi_2} \ar[l]_{g\flat 1} \\
    & X
}
\end{aligned}
\end{equation}
we have $\xi_3=\xi_4$.
\end{itemize}

\begin{Proposition}
Let \C\ be an action representative semi-abelian category (see
\cite{BJK,BB2} or a \emph{category of interest}. Then \C\
satisfies the condition \textnormal{(UA)} above.
\end{Proposition}

\begin{proof}
In an action representative semi-abelian category \C, for every
$X$ in \C, there exists an object $\Act(X)$ (the ``actor'' of
$X$), such that the actions of any object $Y$ on $X$ are in
one-to-one correspondence with the morphisms ${Y \to \Act(X)}$.
Through this correspondence, naming $\phi_i$ the morphism
associated with the action $\xi_i$ in diagram (\ref{diag:C})
above, we get the following commutative diagram:
$$ \xymatrix{
    A \ar[dr]_{\phi_1} \ar[r]^{f} & B \ar@<-.5ex>[d]_(.4){\phi_3} \ar@<.5ex>[d]^(.4){\phi_4}
        & C \ar[dl]^{\phi_2} \ar[l]_{g} \\
    & \Act(X)
} $$
Hence condition \textnormal{(UA)} follows from the fact that $f$ and $g$ are jointly (strongly) epimorphic.

A similar phenomenon occurs in any \emph{category of interest} \C,
that might not be action representative. Nevertheless, it is shown
in \cite{Casas-Datuashvili-Ladra} that, viewing \C\ as a
subvariety of a variety $\C_G$ of groups with operations, for
every object $X$ in \C\ there exists an object $\USGA(X)$ in
$\C_G$ (the ``universal strict general actor'' of $X$), such that
every action of an object $Y$ in \C\ on $X$ yields a morphism $Y
\to \USGA(X)$ in $\C_G$. Then, as above, we get a commutative
diagram in $\C_G$:
$$ \xymatrix{
    A \ar[dr]_{\phi_1} \ar[r]^{f} & B \ar@<-.5ex>[d]_(.4){\phi_3} \ar@<.5ex>[d]^(.4){\phi_4}
        & C \ar[dl]^{\phi_2} \ar[l]_{g} \\
    & \USGA(X)
} $$ The pair $(f,g)$ is jointly strongly epimorphic in $\C_G$
since the same holds in the subvariety \C, and condition (UA)
follows by cancellation.
\end{proof}

Thanks to the previous proposition, the following theorem applies also to the cases of crossed modules of Lie algebras, Leibniz algebras and commutative algebras studied in \cite{CL2000, Aslan, Shammu}.

\begin{Theorem}
Let \C\ be an algebraically coherent semi-abelian category satisfying the condition \textnormal{(UA)}, $(\partial_X \colon X\to B,\xi_X)$ and $(\partial_Y \colon Y \to B,\xi_Y)$ internal crossed modules in \C. Then $(\partial_{\farf}\colon X \farf Y \to B,\xi_{\farf})$, constructed as in Proposition \ref{prop:peiff.prod.ac}, is the coproduct of $\partial_X$ and $\partial_Y$ in $\XMod_B(\C)$.
\end{Theorem}

\begin{proof}
By definition, $(\partial_{\farf},\xi_{\farf})$ is a pre-crossed module. Since in the context we are considering the property \SH\ holds (see \cite{CGrayVdL2}), it suffices to verify the Peiffer identity for $X \farf Y$:
$$ \xymatrix{
    (X \farf Y)\flat(X \farf Y) \ar[rr]^{\partial_{\farf}\flat 1} \ar[dr]_{\chi}
        & & B\flat(X \farf Y) \ar[dl]^{\xi_{\farf}} \\
    & X \farf Y
} $$
In other words, we have to show that the following two actions are equal:
$$ \xymatrix{
    (X \farf Y)\flat(X \farf Y) \ar@<.5ex>[rr]^-{\partial_{\farf}^*\xi_{\farf}} \ar@<-.5ex>[rr]_-{\chi}
        & & X \farf Y \,.
} $$
To prove this, we can pre-compose with the canonical injections $l_X$ and $l_Y$ to obtain the following diagram:
\begin{equation} \label{diag:peiff.farf}
\begin{aligned}
\xymatrix{
    X \flat (X \farf Y) \ar[dr]_{\partial_X^*\xi_{\farf}} \ar[r]^-{l_X\flat 1}
        & (X \farf Y)\flat(X \farf Y) \ar@<.5ex>[d]^(.4){\partial_{\farf}^*\xi_{\farf}} \ar@<-.5ex>[d]_(.4){\chi}
        & Y \flat (X \farf Y) \ar[dl]^{\partial_Y^*\xi_{\farf}} \ar[l]_-{l_Y\flat 1} \\
    & X \farf Y
}
\end{aligned}
\end{equation}
As already observed in Section \ref{sec:definitions}, the pair $(l_X,l_Y)$ is jointly strongly epimorphic in \dense, and then in \C, thanks to condition \CS. Hence, by condition (UA), we only have to prove that diagram (\ref{diag:peiff.farf}) is commutative.

Let us focus on the left hand side triangles. Again by algebraic coherence, the functor $X\flat-$ preserves jointly strongly epimorphic pairs (see \cite{CGrayVdL2}), so the two triangles commute if and only if they commute when pre-composed with the jointly strongly epimorphic pair
$$ \xymatrix{X\flat X \ar[r]^-{1\flat l_X} & X \flat (X \farf Y) & X\flat Y \ar[l]_-{1\flat l_Y} \,.} $$
This is true because, on one hand, $\partial_X$ is a crossed module, then the two squares
$$ \xymatrix{
    X\flat X \ar[d]_{\chi=\partial_X^*\xi_X} \ar[r]^-{l_X\flat l_X}
        & (X \farf Y)\flat(X \farf Y) \ar@<.5ex>[d]^{\partial_{\farf}^*\xi_{\farf}} \ar@<-.5ex>[d]_{\chi} \\
    X \ar[r]_-{l_X} & X \farf Y
} $$
commute because morphisms are equivariant with respect to conjugations and because $l_X$ is a morphism of pre-crossed modules. On the other hand, composing with $1\flat l_Y$ we get the following diagram
$$ \xymatrix{
    X\flat Y \ar[d]_{\partial_X^*\xi_Y} \ar[r]^-{1\flat l_Y} \ar[dr]^(.6){\kappa_{X,Y}}
        & X \flat (X \farf Y) \ar[dr]_{\partial_X^*\xi_{\farf}} \ar[r]^-{l_X\flat 1}
        & (X \farf Y)\flat(X \farf Y) \ar@<.5ex>[d]^{\partial_{\farf}^*\xi_{\farf}} \ar@<-.5ex>[d]_{\chi} \\
    Y \ar[r]_-{\iota_Y} & X+Y \ar[r] & X \farf Y
} $$ where both the outer rectangles are commutative, since the
canonical arrow ${X+Y \to X \farf Y}$ coequalizes the pair
$(\kappa_{X,Y},\iota_Y\cdot\partial_X^*\xi_Y)$, $X \farf Y$ being
a quotient of ${Y \rtimes X}$.

In a symmetric way, one can prove that the right hand side triangles in (\ref{diag:peiff.farf}) commute because $\partial_Y$ is a crossed module and $X \farf Y$ is also a quotient of ${X \rtimes Y}$.
\end{proof}


\end{document}